\documentclass[11pt]{amsart}

\usepackage[table]{xcolor}

\usepackage{ytableau,tikz,varwidth}
\usepackage[all,cmtip]{xy}
\usetikzlibrary{calc, cd}
\usepackage[enableskew]{youngtab}
\usepackage[top=1in, right=1in, left=1in, bottom=1in]{geometry}
\usepackage{enumerate, amssymb, array}
\usepackage[normalem]{ulem}
\usepackage{scalerel}
\usepackage{longtable}
\usepackage{hyperref}

\newtheorem{thm}{Theorem}[section]

\theoremstyle{definition}

\newtheorem{defn}[thm]{Definition}

\newcommand{\sbu}{{\raisebox{1pt}{\scaleto{\bullet}{2.5pt}}}}

\newcommand{\ZZ}{\mathbb{Z}}

\newcommand{\cH}{\mathcal{H}}

\newcommand{\cO}{\mathcal{O}}
\newcommand{\cP}{\mathcal{P}}

\newcommand{\comment}[1]{}

\newcommand{\bs}{\boldsymbol}

\newcommand{\st}{\scriptstyle}

\newcommand{\Fr}{\mathrm{Fr}}
\newcommand{\cpb}{\mathcal{P}^{\sbu}}

\newcommand{\cqb}{\mathcal{Q}^{\sbu}}

\newcommand{\Grd}{G^r_d}

\newcommand{\Grdab}{G^{r,\alpha,\beta}_d}

\newcommand{\hide}[1]{}

\DeclareMathOperator{\Pic}{Pic}

\DeclareMathOperator{\tab}{tab}
\DeclareMathOperator{\vs}{VS}
\DeclareMathOperator{\PT}{PT}

\DeclareMathOperator{\Fl}{Fl}

\newtheorem{Definition}[thm]{Definition}
\newenvironment{definition}
  {\begin{Definition}\rm}{\end{Definition}}

\newtheorem{Example}[thm]{Example}
\newenvironment{example}
  {\begin{Example}\rm}{\end{Example}}

\newtheorem{Fact}[thm]{Fact}
\newenvironment{fact}
  {\begin{Fact}\rm}{\end{Fact}}

\newtheorem{Theorem}[thm]{Theorem}
\newenvironment{theorem}
  {\begin{Theorem}\rm}{\end{Theorem}}
  
\newtheorem{Lemma}[thm]{Lemma}
\newenvironment{lemma}
  {\begin{Lemma}\rm}{\end{Lemma}}

\newtheorem{Remark}[thm]{Remark}
\newenvironment{remark}
  {\begin{Remark}\rm}{\end{Remark}}

\newtheorem{Proposition}[thm]{Proposition}
\newenvironment{proposition}
  {\begin{Proposition}\rm}{\end{Proposition}}

\newtheorem{Corollary}[thm]{Corollary}
\newenvironment{corollary}
  {\begin{Corollary}\rm}{\end{Corollary}}

\theoremstyle{remark}
\newtheorem{obs}[thm]{Observation}

\newcommand \defnow[1]{\begin{definition}{#1}\end{definition}}
\newcommand \exnow[1]{\begin{example}{#1}\end{example}}

\newcommand \enumnow[1]{\begin{enumerate}{#1}\end{enumerate}}

\title{Euler characteristics of Brill-Noether varieties}
\author[M. Chan]{Melody Chan}\address{Department of Mathematics, Brown University, Box
1917, Providence, RI 02912}\email{melody\_chan@brown.edu}

\author[N. Pflueger]{Nathan Pflueger}\address{Department of Mathematics and Statistics, Amherst College, Amherst, MA 01002}\email{npflueger@amherst.edu}
\date{\today}

\begin{document}

\begin{abstract}
We prove an enumerative formula for the algebraic Euler characteristic of Brill-Noether varieties, parametrizing degree $d$ and rank $r$ linear series on a general genus $g$ curve, with ramification profiles specified at up to two general points. Up to sign, this Euler characteristic is the number of standard set-valued tableaux of a certain skew shape with $g$ labels.  We use a flat degeneration via the Eisenbud-Harris theory of limit linear series, relying on moduli-theoretic advances of Osserman and Murray-Osserman; the count of set-valued tableaux is an explicit enumeration of strata of this degeneration. 
\bigskip

\end{abstract}

\maketitle


\section{Introduction}

Fix an algebraically closed field $k$ of characteristic $0$.  Let $X$ be a smooth, proper curve of genus $g$ over $k$, and let $p,q\in X$ be distinct closed points.   Throughout the paper, $r$ and $d$ always denote nonnegative integers, and $\alpha=(\alpha_0,\cdots,\alpha_r)$ and $\beta=(\beta_0,\cdots,\beta_r) \in\ZZ_{\ge0}^{r+1}$ always denote nondecreasing sequences. 

\defnow{Fix $r,d,\alpha,$ and $\beta$ as above.
We write $\Grdab(X,p,q)$ for the {\em moduli space of linear series} of rank $r$ and degree $d$ over $X$, with ramification at least $\alpha$ at $p$ and at least $\beta$ at $q$.} \noindent  We refer to \cite{acgh} for definitions and background, and to \cite{clpt} for more details in the setup that will be most relevant to this paper.

The celebrated Brill-Noether Theorem (first stated and proved without marked points \cite{griffiths-harris-variety} and later extended to curves with one or more marked points \cite{eh-divisors}) concerns the dimension of these varieties: if $(X,p,q)$ is a general twice-marked curve, then when $\Grdab(X,p,q)$ is nonempty, its dimension is given by the Brill-Noether number $\rho$ (see \ref{ss:notation} for the definition of $\rho$ and Theorem \ref{thm:osserman-rhohat} for a precise statement of the Brill-Noether theorem for twice-marked curves). In the case $\rho = 0$, there is an interesting combinatorial version of the Brill-Noether theorem for twice-marked curves (originally due to Castelnuovo in the no-marked-points situation; see \cite[\S 3.1]{tarasca-codim-2} for the situation with marked points or \cite[Theorem 6.3]{clpt} for proofs directly in terms of tableaux): the variety $\Grdab(X,p,q)$ is a union of reduced points, where the number of points is equal to the number of skew standard Young tableaux on a skew shape $\sigma$ that we define below. This paper generalizes this statement to all dimensions (that is, all values of $\rho$), thereby demonstrating an intriguing connection between the geometry of Brill-Noether varieties and combinatorial aspects of the skew shape $\sigma$. Rather than counting points in a $0$-dimensional variety, we compute the algebraic Euler characteristic.

For a proper variety $G$ let $\chi(G)$ denote the algebraic Euler characteristic of $G$, by which we mean the Euler characteristic of the structure sheaf: 
$$\chi(G) = \chi(G, \mathcal{O}_G) = \sum_i (-1)^i h^i(G, \mathcal{O}_G).$$
Our main theorem computes this Euler characteristic of $\Grdab(X,p,q)$.  To state it requires the following definition.

\begin{definition}\label{def:sigma} Fix integers $g,r,d\ge 0$ and $\alpha = (\alpha_0,\ldots,\alpha_r)$ and $\beta = (\beta_0,\ldots,\beta_r) \in \ZZ^{r+1}_{\ge 0}$ nondecreasing sequences.  We let $\sigma = \sigma(g,r,d,\alpha,\beta)$ be the skew Young diagram (Definition~\ref{def:young}) with boxes
$$\{(x,y)\in\ZZ^2 \colon \,0\le y\le r, \,\,-\alpha_y \le x < g-d+r+\beta_{r-y}\}.$$
\end{definition}
\noindent An example is shown in Figure~\ref{fig:sigma}. See Section~\ref{sec:prelim} for preliminaries on Young diagrams and tableaux.

\begin{figure}
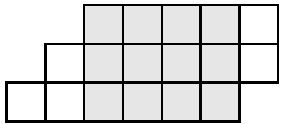

$$
\ytableausetup{boxsize=.5cm} 
\begin{ytableau}
 \none & \none & *(gray!20) &  *(gray!20) & *(gray!20)  & *(gray!20)  &\\
 \none & &  *(gray!20) &  *(gray!20) & *(gray!20)  & *(gray!20)  &  \\
  &  & *(gray!20) &  *(gray!20) & *(gray!20)   & *(gray!20) 
 \end{ytableau}
$$

\caption{The skew Young diagram $\sigma(g,r,d,\alpha,\beta)$ when $g-d=2, r=2, \alpha=(0,1,2),\beta=(0,1,1).$ The leftmost box in the first row is box $(0,0)$.}\label{fig:sigma}
\end{figure}

\begin{thm}\label{thm:main} Let $k$ be an algebraically closed field of characteristic $0$, and fix integers $g,r,d\ge 0$ and $\alpha = (\alpha_0,\ldots,\alpha_r)$ and $\beta = (\beta_0,\ldots,\beta_r) \in \ZZ^{r+1}_{\ge 0}$ nondecreasing sequences.
  For a general twice-pointed smooth, proper curve $(X,p,q)$ over $k$ of genus $g$, 
$$\chi(\Grdab(X,p,q)) = (-1)^{g-|\sigma|}\cdot \#(\text{standard set-valued tableaux on $\sigma$ of content $\{1,\ldots,g\}$}).$$
\end{thm}
\noindent Set-valued tableaux are defined precisely in Definition~\ref{def:young}\eqref{it:std}; they were first introduced by Buch to capture the $K$-theory of the Grassmannian \cite{buch-littlewood-richardson}. 
In fact, the variety is empty if and only if $|\sigma| > g$, which is consistent with the above theorem; see Theorem~\ref{thm:osserman-rhohat}.

It is intriguing that the Euler characteristic in Theorem \ref{thm:main} is (up to sign) given by the answer to an enumerative problem. Indeed, our proof is enumerative in nature: we consider a degeneration of the variety $\Grdab(X,p,q)$ to a reducible variety, whose irreducible components can be enumerated with the aid of tableaux.

Independently, a determinantal formula for the Euler characteristic of twice-pointed Brill-Noether varieties has been established by Anderson-Chen-Tarasca using degeneracy locus methods for maps of vector bundles \cite{anderson-chen-tarasca-k-classes}.  They compute the $K$-class of the structure sheaf of the image of $\Grdab(X,p,q)$ in the Jacobian, from which the Euler characteristic may be computed.  In a separate article \cite{chan-pflueger-combinatorial}, we demonstrate combinatorially that the determinantal formula of \cite{anderson-chen-tarasca-k-classes} is equivalent to the enumerative formula of our paper; see that article for more discussion.

Our proof of Theorem~\ref{thm:main} uses degenerations, limit linear series, and combinatorics.  When a curve degenerates over a suitable one-parameter base to a nodal curve of compact type, the space of linear series on it admits a corresponding degeneration to a space of {\em limit linear series}. That this degeneration may be taken to be flat, under appropriate hypotheses, relies deeply on recent developments of Osserman and Murray-Osserman \cite{murray-osserman, osserman-lls-stack} in the foundations of limit linear series moduli stacks.  The flatness is of course crucial for comparing Euler characteristics of general and special fiber.  The degeneration we use is to an elliptic chain, an idea going back to Welters \cite{welters}.  

The base case of the degeneration therefore involves the linear series on a twice-pointed elliptic curve.  Such spaces are 
examples of relative Richardson varieties defined over the Picard variety of the elliptic curve.  Relative Richardson varieties are defined and studied in-depth in a separate article \cite{chan-pflueger-relative}. Briefly, they are intersections of two Schubert varieties in the relative flag variety $\Fl(\cH)$ of a vector bundle $\cH$, defined with respect to two {\em versal} flag bundles in $\cH$; versality is a condition that generalizes transversality in every fiber. 


To compute $\chi$ of the space of limit linear series on an elliptic chain, we stratify the space and introduce {\em pontableaux} (extending \cite{clpt}), combinatorial objects that encode the strata.  Pontableaux are a generalization of set-valued tableaux.  We compute the Euler characteristics of the strata, together with the M\"obius function on the set of closed strata ordered by containment; what we are relying on is the well-known fact that $\chi$ may be computed by cutting and pasting (Proposition~\ref{prop:cut-n-paste}).

Overall, our approach is aligned with and extends the work of Castorena-L\'opez-Teixidor \cite{castorena-lopez-teixidor}, and our recent joint work with L\'opez and Teixidor \cite{clpt}. The latter article proves Theorem~\ref{thm:main} in the special case $\dim \Grdab(X,p,q) = 1$.  
Incidentally, the translation between the statements of Theorem~\ref{thm:main} and the main theorem in \cite{clpt} is not trivial, due to the introduction of set-valued standard tableaux to this paper.   We hope that the combinatorial understanding of the degeneration of $\Grdab(X,p,q)$ provided in this paper shall provide the foundation for deeper analysis of the geometry of $\Grdab(X,p,q)$.

\section{Preliminaries}\label{sec:prelim}

\subsection{Notation} \label{ss:notation}

We collect here some notation and definitions that will be used throughout the paper.

\begin{eqnarray*}
(X,p,q) && \mbox{is a twice-marked curve of genus $g$.}\\
g,r,d && \mbox{are nonnegative integers.}\\
&& \mbox{(The genus of $X$, and the rank and degree of the linear series in question)}\\
\alpha &=& (\alpha_0, \alpha_1, \cdots, \alpha_r) \mbox{ is a nondecreasing sequence of nonnegative integers.}\\
&& \mbox{(imposed ramification at $p$)}\\
\beta &=& (\beta_0,\beta_1,\cdots,\beta_r) \mbox{  is a nondecreasing sequence of nonnegative integers.}\\
&& \mbox{(imposed ramification at $q$)}\\
\sigma &=& \{(x,y) \in \ZZ^2:\ 0 \leq y \leq r,\ -\alpha_y \leq x < g-d+r + \beta_{r-y}\}\\
&& \mbox{(The skew shape associated to $g,r,d,\alpha,\beta$; see Remark \ref{rem:sigma-is-skew})}\\
\rho &=& g - (r+1)(g-d+r) - \sum_{i=0}^r \alpha_i - \sum_{i=0}^r \beta_i\\
&& \mbox{(the Brill-Noether number)}\\
\hat{\rho} &=& g - |\sigma|\\
&=& g - \sum_{y=0}^r \max(0, \alpha_y + \beta_{r-y} + (g-d+r))\\
&& (\mbox{a test for nonemptiness of $\Grdab(X,p,q)$; see Theorem \ref{thm:osserman-rhohat} (1)})
\end{eqnarray*}

Note that $\rho \geq \hat{\rho}$, with equality if and only if $\alpha_y + \beta_{r-y} \geq - (g-d+r)$ for all $y$. In particular, when $g-d+r \geq 0$, it is always the case that $\rho = \hat{\rho}$.

The curve $X$ will sometimes be a \textit{smooth} curve of genus $g$, and at other times it will be a chain of $g$ elliptic curves. We will specify which is meant in each context.

The notation $\Grdab(X,p,q)$ will be used to refer either to the usual space of limit linear series with imposed ramification (when $X$ is smooth), or to the Eisenbud-Harris space of limit linear series (see Definition \ref{def:ehscheme}).

We will sometimes define other ramification sequences $\alpha^1,\alpha^2,\cdots, \alpha^g$ and $\beta^1,\beta^2,\cdots,\beta^g$ (e.g. in Definition \ref{def:ehscheme}), which should not be confused with the individual elements $\alpha_i,\beta_i$ of the original ramification sequences $\alpha,\beta$. In this situation we will typically require that $\alpha^1 = \alpha$ and $\beta^g = \beta$.

Finally, we warn the reader of a mild abuse of notation in \S\ref{ss:eh-strat}: we will use the symbols $\rho^1, \rho^2, \cdots, \rho^n$ to denote sequences of integers (whose elements are $\rho^n_i$), which are not related to the Brill-Noether number $\rho$ (with no subscripts of superscripts). 

\subsection{Limit linear series} \label{sec:prelim-lls}

We recall some preliminaries on the theory of Eisenbud-Harris {limit linear series} \cite{eh-lls}, following \cite{clpt}. This theory applies to all reduced, nodal curves of compact type; but we describe here only the case of elliptic chains.  A twice-marked {\em elliptic chain} $(X,p,q)$ of genus $g$ is a proper, reduced, nodal curve $X$ obtained by taking twice-marked genus 1 curves $(E_1,p = p_1,q_1),\ldots, (E_g,p_g,q_g = q)$ and gluing $q_{i}$ to $p_{i+1}$ nodally.  We say that $(X,p,q)$ is 
{\em generic} if $p_i-q_i$ is not torsion in $\Pic^0(E_i)$.

\begin{definition} \label{def:ehscheme}
Let $(X,p,q)$ be a twice-marked elliptic chain.
The {\em Eisenbud-Harris scheme of limit linear series} $\Grdab(X,p,q)$ is the subscheme of $\prod_{i=1}^g \Grd(E_i)$ obtained as a union
$$\bigcup \prod_{i=1}^g G_d^{r,\alpha^i,\beta^i} (E_i, p_i , q_i).$$
The union above ranges over choices of ramification profiles $(\alpha^i,\beta^i)_{i=1}^g$
with $\alpha^1 = \alpha$ and $\beta^g = \beta$, such that
$$\beta^i_j + \alpha^{i+1}_{r-j} \ge d-r,$$ for each $j=0,\ldots,r$ and each $i=1,\ldots,g-1$.
\end{definition}
The $k$-points of the scheme $\Grdab(E,p,q)$ correspond to isomorphism classes of {\em limit linear series} on the elliptic chain $(X,p,q)$. A limit linear series may be described as a $g$-tuple $(L_1,L_2,\cdots,L_g)$, where $L_i$ is a linear series on the elliptic curve $E_i$ called the {\em $E_i$-aspect}. A limit linear series is called {\em refined} if, denoting the ramification sequence of the $E_i$-aspect at $p_i$ by $\alpha^i$ and at $q_i$ by $\beta^i$, the equation
$$
\beta^i_j + \alpha^{i+1}_{r-j} = d-r
$$
holds for all $1 \leq i \leq g-1$ and $0 \leq j \leq r$. A limit linear series is called {\em coarse} if it is not refined.

\subsection{Set-valued tableaux}\label{ss:svt}


Fix the partial order $\preceq$ on $\ZZ^2$ given by $(x,y)\preceq (x',y')$ if $x\le x'$ and $y\le y'$.  

\begin{definition}\label{def:young}\mbox{}
\begin{enumerate}
\item 
A skew Young diagram is a finite subset $\sigma \subset \ZZ^2$ that is closed under taking intervals.  In other words, $\sigma$ has the property that if $(x,y)$ and $(x',y')\in  \sigma$ with $(x,y)\preceq (x',y')$, then $$\{(x'',y''):(x,y)\preceq (x'',y'') \preceq (x',y')\}\subseteq \sigma.$$
\item A skew Young diagram is called a {\em Young diagram} if $\sigma$ has a unique minimal element.
\end{enumerate}
\end{definition}

Skew Young diagrams are sometimes also called {\em skew shapes}, and skew Young diagrams having a unique minimal element will sometimes be called {\em straight shapes} for emphasis. 
In accordance with the English notation for Young diagrams, we will draw the points of $\ZZ^2$ arranged with $x$-coordinate increasing from left to right, and $y$-coordinate {\em increasing} from top to bottom, e.g.
$$
\begin{array}{ccc}
(0,0) & (1,0) & \cdots \\
(0,1) & (1,1) & \\
\vdots & &
\end{array}
$$
Furthermore, we will draw, and refer to, the members of $\sigma$ as \emph{boxes}, and we let $|\sigma|$ denote the number of boxes in $\sigma$.  We do not require that $\sigma$ be connected. 


\begin{remark}\label{rem:sigma-is-skew}
Let $g,r,d,\alpha,\beta$ and $\sigma$ be as in \S\ref{ss:notation}. It follows from $\alpha, \beta$ being nondecreasing that $\sigma$ is indeed a skew shape. The number of boxes is
$$|\sigma| = \sum_{y=0}^r \max\left(0,\alpha_y + \beta_{r-y} +  (g\!-\!d\!+\!r)\right).$$
\end{remark}

\begin{definition} \label{def:tableau}
A {\em tableau} of shape $\sigma$ is an assignment $T$ of a positive integer, called a {\em label}, to each box of $\sigma$.  \begin{enumerate}
\item A tableau $T$ of shape $\sigma$ is {\em semistandard} if the rows of $\sigma$ are weakly increasing from left to right, and the columns of $\sigma$ are strictly increasing from top to bottom.
\item A tableau $T$ of shape $\sigma$ is {\em standard} if it is semistandard and furthermore each integer $1,\ldots,|\sigma|$ occurs exactly once as a label.
\end{enumerate}

\end{definition}

\begin{definition}\cite{buch-littlewood-richardson} \label{def:svt}
A {\em set-valued tableau} of shape $\sigma$ is an assignment of a nonempty finite set of positive integers  to each box of $\sigma$. 

Given sets $S,T\subseteq \ZZ_{>0}$, we write $S<T$ if $\max(S)<\min(T)$, and we write $S\le T$ if  $\max(S)\le\min(T)$.  Then we extend the definitions of {\em semistandard} and {\em standard} tableaux to set-valued tableaux.
\begin{enumerate}
\item \label{it:sstd}A set-valued tableau $T$ of $\sigma$ is {\em semistandard} if the rows of $\sigma$ are weakly increasing from left to right, and the columns of $\sigma$ are strictly increasing from top to bottom.
\item \label{it:std} A set-valued tableau $T$ of $\sigma$ {\em standard} if it is semistandard and furthermore the labels are pairwise disjoint sets with union $\{1,\ldots,r\}$ for some $r \ge |\sigma|$.
\end{enumerate}
\end{definition}
The \emph{content} of a set-valued tableau is the multiset of labels occurring in it. So a semistandard tableau is standard if and only if the content is $\{1,2,\cdots,|\sigma|\}$. We call a tableau \emph{almost-standard} if it is semistandard and no label occurs more than once (i.e. all elements of the content have multiplicity $1$).

\section{Linear series on an elliptic curve}
\label{sec:ls-elliptic}

Fix integers $g,d,r\ge 0$, and nondecreasing sequences $\alpha$ and $\beta\in \ZZ_{\ge0}^{r+1}$.  
In this section, we establish basic results on the geometry of the moduli space $\Grdab(X,p,q)$, especially in the case where $X$ has genus $1$.

We record the following basic facts about the nonemptiness, dimension, reducedness, and irreducibility of the variety $\Grdab(X,p,q)$. See \S\ref{ss:notation} for definitions of $\rho$ and $\hat{\rho}$. Part (1) of this theorem is what is usually called the Brill-Noether theorem for curves with two marked points\footnote{Often the phrase ``Brill-Noether theorem'' refers specifically the the case where there are no marked points and no ramification is imposed (as it was originally formulated), while this version, or a generalization to more marked points in characteristic zero, is called the ``extended Brill-Noether theorem.''}. It is discussed in \cite{eh-divisors} with a characteristic $0$ hypothesis, while \cite{osserman-simple} gives a characteristic-free argument.

\begin{thm}\cite[Theorem 1.1, Lemma 2.1]{osserman-simple}\label{thm:osserman-rhohat}
Fix $g,r,d,\alpha,$ and $\beta.$ 

\enumnow{
\item\label{it:nonempty} For a general twice-pointed curve $[(X,p,q)]\in \mathcal{M}_{g,2}$, the variety $\Grdab(X,p,q)$ is nonempty if and only if $\hat{\rho}\ge 0$. If it is nonempty, then it has pure dimension $\rho$.
\item\label{it:nonempty-E} When $X=E$ is a curve of genus 1, to ensure that statement~\eqref{it:nonempty} holds, it suffices to take points $p,q\in E$ such that $p-q$ is not $d'$-torsion in $\Pic^0(E)$ for any $d'\le d$.  Furthermore, in this situation $\Grdab(E,p,q)$ is reduced and irreducible.
}
\end{thm}
In fact, Theorem~\ref{thm:osserman-rhohat} is also proved by Osserman using  a degeneration to elliptic chains, so it is reasonable that our argument is closely related.  In fact, the analysis of the skew shape $\sigma(g,r,d,\alpha,\beta)$ captures the combinatorial part of the analysis in \cite{osserman-simple}.  Another proof of Theorem~\ref{thm:osserman-rhohat}(1) is given in \cite{anderson-chen-tarasca-k-classes} as well.

\begin{proposition}\label{prop:int-reduced}
Given nondecreasing sequences $\alpha_1,\beta_1,\alpha_2,\beta_2\in \ZZ_{\ge0}^{r+1}$, let
$$\alpha = \max \alpha_1,\alpha_2 \quad\text{and}\quad \beta=\max\beta_1,\beta_2.$$
Then
$$G^{r,\alpha_1,\beta_1}_d(X,p,q)\,\cap\, G^{r,\alpha_2,\beta_2}_d(X,p,q) = \Grdab(X,p,q).$$
In particular, under the hypotheses of Theorem~\ref{thm:osserman-rhohat}(2), every such intersection is a reduced scheme.
\end{proposition}

\begin{proof}
The equality is clear on the level of sets.  The scheme theoretic equality follows from the description of the functor of points of $\Grdab(X,p,q)$ in \cite[Theorem 4.1.3]{osserman-book}.
\end{proof}

We now specialize to the case where $X=E$ is a curve of genus 1, and $p,q\in E$ such that $p-q$ is not $d'$-torsion for any $d'\le d$.  Recall that there is a natural projection morphism
$$\pi\colon \Grdab(E,p,q)\to\Pic^d(E),$$
described on points by forgetting the sections of a linear series.



\begin{lemma} \label{lemma:when-surjective-to-Pic}
The morphism $\pi: G \rightarrow \Pic^d E$ is surjective if and only if $\hat{\rho} = 1$. Equivalently, this holds if and only if $\alpha_y + \beta_{r-y} < d-r$ for all $y$.
\end{lemma}

\begin{proof}
This is implicit in the proof of Lemma 2.1 in \cite{osserman-simple}.
\end{proof}

\begin{proposition} \label{prop:cohomology-of-grdab} Write $G=\Grdab(E,p,q)$ for short. The morphism $\pi \colon G\to \Pic^d E$ is either surjective or has image a point.  Moreover:
\enumnow{
\item
In the first case, we have 
$$H^i(G,\cO_G)\cong\begin{cases}
k &\text{if }i=0,1,\\
0 &\text{else.}
\end{cases}$$
In particular if $\pi$ is surjective then $\chi(G) = 0$.
\item
In the second case, we have 
$$H^i(G,\cO_G)\cong\begin{cases}
k &\text{if }i=0,\\
0 &\text{else.}
\end{cases}$$
In particular if $\pi$ has image a point then $\chi(G) = 1$.
}\end{proposition}

\begin{proof}
This is proved in \cite{chan-pflueger-relative}; we summarize the main points as follows. 

Given a rank-$d$ vector bundle $\cH$ over a smooth irreducible base scheme $S$, two complete flags $\cpb$ and $\cqb$ in $\cH$ are called {\em versal} if the map $\Fr(\cH)\to \Fl(d)^2$ that they define, where $\Fr(\cH)$ denotes the frame bundle, is a smooth morphism. One may consider the relative partial flag variety $\Fl(i_0, \ldots, i_s; \cH)$ parameterizing flags of subspaces of codimension $i_0, \ldots, i_s$. Within this relative flag variety are relative Schubert varieties, defined in terms of either $\cpb$ or $\cqb$. An intersection of two relative Schubert varieties inside a relative partial flag variety $\Fl(i_0,\ldots,i_s;\cH)$, defined with respect to $\cpb$ and $\cqb$ respectively, is called a {\em relative Richardson variety.}  In our situation, $\cH$ is the rank-$d$ vector bundle over $\Pic^d(E)$ whose fiber over $[L]$ is canonically identified with $H^0(E,L)$, and $\cpb$ and $\cqb$ are the flags determined by vanishing orders at $p$ and $q$, respectively.

Of the many geometric properties enjoyed by a relative Richardson variety $\pi\colon R\to S$, as proved in \cite{chan-pflueger-relative}, one of them is that
$H^i(R,\cO_R) \cong H^i(S',\cO_{S'})$ for all $i\ge0$, where $S' = \pi(R)$ is the scheme-theoretic image of $R$ in $S$.  In particular, the Proposition follows.
\end{proof}

\section{Limit linear series on an elliptic chain}

The objective of this section is to analyze the scheme structure of the Eisenbud-Harris scheme of limit linear series $\Grdab(X,p,q)$, where $(X,p,q)$ is a generic twice-marked elliptic chain (as defined in \ref{sec:prelim-lls}), and to deduce Theorem \ref{thm:main} from this analysis.

The main results about the structure of $\Grdab(X,p,q)$ for an elliptic chain are the following; these will be proved in \S\ref{ss:ehscheme-proofs} after preliminary results are established in \S\ref{ss:eh-strat} and \S\ref{sec:pontableaux}.

The first of these, Theorem \ref{thm:grdab-chain-structure}, is not new. It follows, for example, from the proof of Theorem 1.1 in \cite{osserman-simple}; that proof uses an inductive argument in which an elliptic curve is attached to a genus $g-1$ smooth curve, which adapts readily to constructing a chain of $g$ elliptic curves. But we are not aware of a reference stating this fact specifically for an elliptic chain, so we state the result for convenience and will provide a brief proof in our notation. 

\begin{thm} \label{thm:grdab-chain-structure}
Let $(X,p,q)$ be a generic twice-marked elliptic chain, as defined in \S\ref{sec:prelim-lls}. Then the Eisenbud-Harris scheme $\Grdab(X,p,q)$ is nonempty if and only if $\hat{\rho} \geq 0$. If nonempty, the Eisenbud-Harris scheme is reduced of dimension $\rho$, and the locus of refined series is dense. (Here $\rho,\hat{\rho}$ are as in \S \ref{ss:notation}.)
\end{thm}

\begin{thm} \label{thm:grdab-chain-euler}
Let $\Grdab(X,p,q)$ be as in Theorem \ref{thm:grdab-chain-structure}. Then
$$
\chi(\Grdab(X,p,q)) = (-1)^{g-|\sigma|} \cdot \# \left(\mbox{standard set-valued tableaux on $\sigma$ of content } \{1,\ldots,g\}\right)
$$
\end{thm}

In order to deduce Theorem \ref{thm:main} from Theorems \ref{thm:grdab-chain-structure} and \ref{thm:grdab-chain-euler}, we will consider a smoothing of $X$ to a smooth curve of genus $g$, and make use of results from the theory of limit linear series. This is carried out in \S\ref{ss:mainthm-proof}.

\subsection{The stratification of the Eisenbud-Harris space} \label{ss:eh-strat}

Throughout this subsection, fix data $(g,r,d,\alpha,\beta)$, and let $(X,p,q)$ be a generic twice-marked elliptic chain of genus $g$, as defined in \S\ref{sec:prelim-lls}. We will describe the irreducible components of the Eisenbud-Harris scheme $\Grdab(X,p,q)$. We begin by reviewing some definitions and facts from \cite{clpt}.

Let $\bs \alpha = (\alpha^1,\cdots,\alpha^{g+1})$ denote a $(g+1)$-tuple of nondecreasing sequences of $r+1$ integers. We call $\bs \alpha$ a \emph{valid sequence for the data $(g,r,d,\alpha,\beta)$} if the following three conditions hold (\cite[Definition 3.8]{clpt}).
\begin{enumerate}
\item For $i=0,\cdots,r$, $\alpha^1_i = \alpha_i$.
\item For $i=0,\cdots,r$, $\alpha^{g+1}_i = d-r-\beta_{r-i}$.
\item For $n=1,\cdots,g$ and $i=0,\cdots,r$,
$$
\alpha^{n+1}_j \geq \alpha^n_j,
$$
where, for any particular value of $n$, equality holds for at most one value of $j$ (which may be different for different values of $n$).
\end{enumerate}

The set of all valid sequences for the data $(g,r,d,\alpha,\beta)$ is denoted $\vs(g,r,d,\alpha,\beta)$.

Given a valid sequence $\bs \alpha$, define the \emph{complementary sequence} $\bs \beta = (\beta^0,\cdots,\beta^g)$ by
$$
\beta^n_i = d-r - \alpha^{n+1}_{r-i}.
$$
As in \cite[Definition 4.6]{clpt}, define
$$
C({\bs \alpha}) = \prod_{n=1}^g G^{r,\alpha^n,\beta^n}_{d}(E_n,p_n,q_n) \subseteq \prod_{n=1}^g G^r_d(E_n).
$$

The definition of the complementary sequence guarantees that any point $(L_1,\cdots,L_g) \in C({\bs \alpha})$ may be regarded as the aspects of a limit linear series, and the first two conditions in the definition of a valid sequence ensure that this limit linear series lies in $\Grdab(X,p,q)$. In fact, the loci $C({\bs \alpha})$ constitute a decomposition of $\Grdab(X,p,q)$ into irreducible components.

\begin{lemma}{\cite[Corollary 4.8]{clpt}} \label{lemma:vsunion}
For any choice of data $(g,r,d,\alpha,\beta)$ and generic twice-marked elliptic chain $(E,p,q)$ of genus $g$,
$$\Grdab(X,p,q) = \bigcup_{{\bs \alpha} \in \vs(g,r,d,\alpha,\beta)} C({\bs \alpha}).$$
Furthermore, the locus of refined limit linear series is equal to the set of points lying in just one of the schemes $C(\bs \alpha)$.
\end{lemma}

It is necessary for our purposes to enumerate all intersections of the loci $C(\bs \alpha)$ as well. To do so, we introduce the following terminology. Let $\bs \alpha$ be a valid sequence. A sequence $\bs \beta = (\beta^1,\cdots,\beta^g)$ of nondecreasing $(r+1)$-tuples is called a \emph{compatible sequence} for $\bs \alpha$ if for all $n,i$,
$$
\beta^n_i \geq d-r-\alpha^{n+1}_{r-i}.
$$
For a valid sequence $\bs \alpha$ and compatible sequence $\bs \beta$, define
$$
C(\bs \alpha, \bs \beta) = \prod_{n=1}^g G^{r,\alpha^n,\beta^n}_{d}(E_n,p_n,q_n) \subseteq \prod_{n=1}^g G^r_d(E_n).
$$
The locus $C(\bs \alpha)$ is a special case, in which $\bs \beta$ is taken to be the complementary sequence. The loci $C(\bs \alpha, \bs \beta)$ include all intersections of any set of loci $C(\bs \alpha)$, due to the following.

\begin{lemma} \label{lem:cab-intersections}
For valid sequences $\bs \alpha, \bs \alpha'$ and sequences $\bs \beta, \bs \beta'$ compatible with them (respectively),
$$
C(\bs \alpha, \bs \beta) \cap C(\bs \alpha', \bs \beta') =
C(\max(\bs \alpha,\bs \alpha'), \max(\bs \beta, \bs \beta')),
$$
scheme-theoretically. Here by $\max(\bs \alpha,\bs \alpha')$ the sequence of $(r+1)$-tuples formed by taking the maximum of each element of each sequence individually.
\end{lemma}
\begin{proof}
This follows from Proposition \ref{prop:int-reduced}, applied to each factor individually in the definition of $C(\bs \alpha)$. 
\end{proof}

Note that in Lemma \ref{lem:cab-intersections}, it is not necessarily true that $\max(\bs \alpha, \bs \alpha')$ is again a valid sequence; if not, the intersection will be empty.

The attributes of the loci $C(\bs \alpha, \bs \beta)$ needed in our analysis are summarized as follows.

\begin{lemma} \label{lemma:cab-attributes}
Let $\bs \alpha$ be a valid sequence for data $(g,r,d,\alpha,\beta)$, and let $\bs \beta$ be a compatible sequence.
\begin{enumerate}
\item $C(\bs \alpha, \bs \beta)$ is nonempty if and only if for all $n \in \{1,2,\cdots,g\}$, $i \in \{0,\cdots,r\}$,
$$
\alpha^n_i + \beta^n_{r-i} \leq d-r,
$$
with equality for at most one value of $i$ per value of $n$.
\item If $C(\bs \alpha, \bs \beta)$ is nonempty, then it is reduced and equidimensional with
$$
\dim C(\bs \alpha, \bs \beta) = \rho - \sum_{n=1}^{g-1} \sum_{i=0}^r \left( \beta^n_i + \alpha^{n+1}_{r-i} - (d-r) \right).
$$
\item If $C(\bs \alpha, \bs \beta)$ and $C(\bs \alpha', \bs \beta')$ are nonempty, then the containment $C(\bs \alpha, \bs \beta) \subseteq C(\bs \alpha', \bs \beta')$ holds if and only if for all $n,i$, $\alpha^n_i \geq \alpha'^n_i$ and $\beta^n_i \geq \beta'^n_i$.
\item If $\bs \beta$ is complementary to $\bs \alpha$, then a dense open subset of $C(\bs \alpha, \bs \beta)$ consists of refined limit linear series. Otherwise, all points correspond to coarse series.
\item The Euler characteristic of the structure sheaf is given by
$$
\chi \left( C(\bs \alpha, \bs \beta) \right) = \begin{cases}
1 & \mbox{ if for all $n$, there is some $i$ with equality } \alpha^n_i + \beta^n_{r-i} = d-r,\\
0 & \mbox{ otherwise.}
\end{cases}
$$
\end{enumerate}
\end{lemma}
\begin{proof}
Parts (1) and (2) follows from Theorem \ref{thm:osserman-rhohat}, applied to each elliptic curve $E_i$ individually. One direction of part (3) follows from Lemma \ref{lem:cab-intersections}, while the converse follows from part (2): if the stated inequalities do not hold, then the intersection of the two loci would have dimension strictly smaller than either locus. Part (4) follows by observing that a limit linear series is refined if and only if it lies in $C(\bs \alpha, \bs \beta)$ for a complementary choice of $\bs \alpha, \bs \beta$, but not in $C(\bs \alpha', \bs \beta')$ for any other choice of $\alpha', \beta'$ (see Lemma \ref{lemma:vsunion}), together with the fact (from part (2)) that any other locus would intersect $C(\bs \alpha,\bs \beta)$ (which is equidimensional) in a locus of strictly smaller dimension. 

Part (5) follows from Proposition~\ref{prop:cohomology-of-grdab}, together with the fact that for each $n$, the morphism $G^{r,\alpha^n,\beta^n}_d(E_n, p_n, q_n) \rightarrow \Pic^d(E_n)$ is surjective if any only if there is no $i$ such that $\alpha^n_i + \beta^n_{r-i} = d-r$ (Lemma \ref{lemma:when-surjective-to-Pic}).
\end{proof}

In \S\ref{sec:pontableaux}, we will describe a convenient way to enumerate the pairs $\bs \alpha, \bs \beta$ giving nonempty strata $C(\bs \alpha, \bs \beta)$ and completing the proof of Theorems \ref{thm:grdab-chain-structure} and \ref{thm:grdab-chain-euler}.
But first we recall a standard fact relating the Euler characteristic of a union with the Euler characteristics of the irreducible components and their intersections.  
\begin{definition}\label{def:our-mobius}
For any finite poset $\cP$, define a M\"obius function $\mu_\cP$ on $\cP$ as follows: if $Z$ is maximal, 
let $\mu_\cP(Z)=1$.  Otherwise, define $\mu_\cP(Z)$ recursively by
$$\mu_\cP(Z) = 1 - \sum_{\{Y\in \cP: Y>Z\}} \mu_\cP(Y).$$
\end{definition}
\noindent (This differs from the usual M\"obius function, see e.g.~\cite[\S3.7]{ec1}, by a minor change in convention.)
Now let $Z_1,\ldots,Z_N$ be irreducible closed subschemes of a projective $k$-scheme $X$, and let $G=Z_1\cup \cdots\cup Z_N$ be the scheme-theoretic union.  By a {\em closed stratum} of $G$ we mean any  nonempty intersection of the $Z_i$.  We suppose every closed stratum (including the $Z_i$ themselves) is reduced.  We let $\cP$ be the poset of closed strata, ordered by inclusion.

\begin{proposition}\label{prop:cut-n-paste}
We have $$\chi(G) =\sum_{Z\in\cP} \mu_\cP(Z) \chi(Z).$$
\end{proposition}

\begin{proof}
This follows, using standard combinatorics, from the elementary fact that $\chi(Z_i\cup Z_j) = \chi(Z_i)+\chi(Z_j)-\chi(Z_i\cap Z_j).$ 
\end{proof}

\subsection{The poset of pontableaux} \label{sec:pontableaux}

The irreducible components of the space of limit linear series on an elliptic chain, and all intersections thereof, can be enumerated by combinatorial objects, closely analogous to set-valued tableaux, called {\em pontableaux}; the pontableaux of \cite{clpt} are a special case (see Remark \ref{rem:pontableauxCLPT}).

In what follows, we will define, for data $(g,r,d,\alpha,\beta)$ as in Definition \ref{def:sigma}, a poset $\PT(g,r,d,\alpha,\beta)$, whose elements are called pontableaux, which will be in bijection with the strata of the space of limit linear series on an elliptic chain. We will define combinatorial attributes $\mu(P), \dim P, \chi(P)$ for all pontableaux $P$; we will show (Lemma \ref{lemma:ct-attributes}) that these coincide with the M\"obius function, dimension, and Euler characteristic of the corresponding strata. After making these definitions, the main combinatorial result of this section will be the following theorem.

\begin{theorem} \label{prop:ptmuchi}
Given data $(g,r,d,\alpha,\beta)$, let $\sigma$ be the corresponding skew shape (Definition \ref{def:sigma}). Then 
$$
\sum_{P \in \PT (g,r,d,\alpha,\beta)}\!\!\!\!\!\!\!\!\!\!\!\!\mu(P) \chi(P) = (-1)^{g - |\sigma|} \cdot \# \left( \mbox{standard set-valued tableaux on $\sigma$ of content } \{1,\ldots,g\} \right).
$$
\end{theorem}

\subsubsection{The set of pontableaux} \label{sec:pontSet}

Pontableaux will be defined in terms of sequences of nonincreasing $(r+1)$-tuples of integers. These $(r+1)$-tuples may be interpreted as 
the right border of a set of boxes extending
infinitely to the left. We first fix some notation that will be convenient throughout this section.

\begin{enumerate}
\item Lowercase Greek letters will denote non-increasing $(r+1)$-tuples of integers, which will be indexed from $0$ to $r$. The elements of a tuple $\lambda$ will be denoted $(\lambda_0,\cdots,\lambda_r)$. Note that there is no requirement that the $\lambda_i$ be nonnegative.  We will identify a tuple $\lambda$ with the set $\{ (x,y):\ 0 \leq y \leq r,\ x < \lambda_y\}$. For example, we will write $\lambda \subseteq \rho$ to mean that $\lambda_y \leq \rho_y$ for all $0 \leq y \leq r$. The elements of this set will be called the boxes \emph{contained in} the tuple.
Visually, $\lambda$ defines an eastern border on $\ZZ^2$ and we associate to $\lambda$ the infinite set of boxes to the left of the border.
\item If $\rho,\lambda$ are two nonincreasing $(r+1)$-tuples, then $\rho / \lambda$ will denote the skew shape obtained by taking the set difference of the boxes contained in $\rho$ minus the boxes contained in $\lambda$.
$$
\rho / \lambda = \{ (x,y):\ 0 \leq y \leq r,\ \lambda_y \leq x < \rho_y \}.
$$
Note that we do not necessarily assume that $\lambda_y \subseteq \rho_y$ when using this notation.
\item We restrict our attention to the horizontal strip $\ZZ\times\{0,1,\ldots,r\}$, whose elements are called  \emph{boxes}.
We will say that a box $(x,y)\in \ZZ \times \{0,1,\ldots,r\}$ is an {\em inward corner}, respectively an {\em outward corner}, of $\lambda$ if it is minimal not in $\lambda$, respectively maximal in $\lambda$, with respect to the order $\preceq$ on $\ZZ$ (Definition~\ref{def:young}).

\item If $(x,y)$ is an inward corner, we will write $\lambda \cup (x,y)$ to denote the tuple resulting from increasing $\lambda_y$ by $1$. If $(x,y)$ is an outward corner, we will write $\lambda \backslash (x,y)$ to denote the tuple resulting from decreasing $\lambda_y$ by one.
\end{enumerate}

\exnow{Let $\lambda = (3,1)$. The inward corners of $\lambda$ are $(3,0)$ and $(1,1)$, and the outward corners of $\lambda$ are $(2,0)$ and $(0,1)$.
$$\cdots\Yvcentermath1\young(\ \ \ ,\ )$$
}

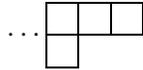
\begin{obs}
A box $(x,y)$ is an inward corner of $\lambda$ if and only if it is an outward corner of $\lambda \cup (x,y)$. Every tuple $\rho$ that differs from $\lambda$ by $1$ in one place is obtained by either adding an inward corner or removing an outward corner of $\lambda$.
\end{obs}

\begin{defn} \label{def:pontseq}
A \emph{pontableau sequence} is a sequence $t = (\lambda^1,\rho^1,\lambda^2,\rho^2,\cdots,\lambda^g,\rho^g)$ of non-increasing $(r+1)$-tuples of integers, satisfying the following two conditions.
\begin{enumerate}
\item For all $1 \leq n \leq g-1$, $\rho^n \supseteq \lambda^{n+1}$.
\item For all $1 \leq n \leq g$, either $\lambda^n \supseteq \rho^n$ or there is a single inward corner $b_n$ of $\lambda^n$ such that $\lambda^n \cup b_n \supseteq \rho^n$.
\end{enumerate}

Denote by $\PT(\lambda^1,\rho^g)$ the set of pontableaux with the specified values of $\lambda^1$ and $\rho^g$. For data $(g,r,d,\alpha,\beta)$ as in \S \ref{ss:notation}, denote by $\PT(g,r,d,\alpha,\beta)$ the set $\PT(\lambda^1,\rho^g)$, where
\begin{eqnarray*}
\lambda^1 &=& (-\alpha_0, -\alpha_1, \cdots, -\alpha_r)\\
\rho^g &=& (g-d+r + \beta_r, g-d+r + \beta_{r-1}, \cdots,g-d+r + \beta_0).
\end{eqnarray*}
\end{defn}

In other words, a pontableau sequence is a sequence of box sets that can grow only one box at at time, and only between $\lambda^n$ and $\rho^n$, but that can shrink by any number of boxes at any step.

\begin{defn} \label{def:pontramseq}
Given a pontableau sequence $t = (\lambda^1,\rho^1,\lambda^2,\rho^2,\cdots,\lambda^g,\rho^g)$, the \emph{ramification sequences} of $t$ are the following nondecreasing $(r+1)$-tuples, for $1 \leq n \leq g$.
\begin{eqnarray*}
\alpha^n_i &=& (n-1) - \lambda^n_i\\
\beta^n_i &=& \rho^n_{r-i} - (n-d+r)
\end{eqnarray*}
\end{defn}

\begin{obs} \label{obs:pontcompat}
The conditions defining a pontableau sequence are equivalent to the following conditions on the associated ramification sequences.
\begin{enumerate}
\item $\alpha^{n+1}_i + \beta^{n}_{r-i} \geq d-r$ for all $i \in \{0,1,\cdots,r\}$.
\item $\alpha^n_i + \beta^n_{r-i}  \leq d-r$ for all $i \in \{0,1,\cdots,r\}$, with equality for at most one value of $i$.
\end{enumerate}
These conditions are equivalent to saying, in the language of \S \ref{ss:eh-strat}, that the $\alpha^n$ form a valid sequence (where we define $\alpha^{g+1}$ in terms of $\beta$ by $\alpha^{g+1}_i = d-r-\beta_{r-i}$), the $\beta^n$ form a compatible sequence, and these sequences satisfy the necessary and sufficient condition of Lemma \ref{lemma:cab-attributes}(1) to determine a nonempty stratum.
\end{obs}

Although pontableau sequences are convenient for use in our proofs, it is preferable to concisely encode the same information as the sequence, as follows.

Given a pontableau sequence $(\lambda^1,\cdots,\rho^g)$, we associate to each $(x,y) \in \ZZ \times \{0,1,\cdots,r\}$ a set $t(x,y)$ of symbols. Each symbol is one of ``$n$'', ``$-n$,'' or ``$n-$,'' where $n$ is an integer from $\{1,2,\cdots,g\}$. The set $t(x,y)$ is determined as follows.
\begin{enumerate}
\item If $(x,y)$ is contained in $\rho^n$ but not $\lambda^n$, then include the symbol ``$n$'' in $t(x,y)$. This symbol is called an \emph{augmentation}.
\item If $(x,y)$ is contained in $\rho^{n-1}$ but not $\lambda^{n}$, then include the symbol ``$-n$'' in $t(x,y)$. This symbol is called a \emph{left removal}.
\item If $(x,y)$ is contained in $\lambda^n$ but not $\rho^n$, then include the symbol ``$n-$'' in $t(x,y)$. This symbol is called a \emph{right removal}.
\end{enumerate}

Taken together, the sets $t(x,y)$ and the initial tuple $\lambda^1$ uniquely encode all of the tuples in a pontableau sequence. Indeed, the left removals $-n$ encode which boxes must be removed from $\rho^{n-1}$ to obtain $\lambda^n$, and the right removals $n-$ and augmentation $n$ (which occurs in a unique box if at all) encode which boxes must be removed and added from $\lambda^n$ to obtain $\rho^n$, respectively.

\begin{defn} \label{def:pontlabeling}
The \emph{labeling associated to a pontableau sequence} $t = (\lambda^1,\cdots,\rho^g)$ is obtained by writing in every box $(x,y) \in \ZZ \times \{0,1,\cdots,r\}$ the elements of the set $t(x,y)$ as described above. We will use the word \emph{pontableau} to refer interchangeably to a pontableau sequence or to its associated labeling.
\end{defn}

\begin{defn} \label{def:underlyingSVT}
For a pontableau $P$, the \emph{underlying set-valued tableau}, denoted $\tab (P)$, is the skew set-valued tableau obtained by placing in box $(x,y)$ all of the augmentations ``$n$'' in $t(x,y)$.
\end{defn}

\begin{example} \label{ex:pontableau}
Let $(g,r,d,\alpha,\beta) = (2,1,4,(0,0),(0,2))$. We consider the pontableau sequence $t=(\lambda^1,\rho^1,\lambda^2,\rho^2)= ((0,0),(1,-1),(1,-2),(1,-1)).$
The associated ramification sequences are recorded in the table below.
$$
\begin{array}{|r|cc|cc|}
\hline
i & \alpha^1 & \beta^1 & \alpha^2 & \beta^2\\
\hline
0 & 0 & 1 & 0 & 0\\
1 & 0 & 3 & 3 & 2\\
\hline
\end{array}
$$
The pontableau labeling associated to $t$ is

$$
\begin{array}{|c|c|c|}
\cline{3-3}
\multicolumn{2}{c|}{} & 1\\
\cline{1-3}
-2,2 & 1- & \multicolumn{1}{c}{}\\
\cline{1-2}
\end{array}
$$

\noindent and the underlying set-valued tableau of $t$  is
$$
\begin{array}{|c|c c|}
\cline{3-3}
\multicolumn{2}{c|}{} & 1\\
\cline{1-1} \cline{3-3}
2 & \,\,\, & \multicolumn{1}{c}{}\\
\cline{1-1}
\end{array}
$$

\end{example}
\smallskip

\begin{remark} \label{rem:pontableauxCLPT}
The ``pontableaux'' of \cite{clpt} are a special case of the pontableaux we have defined here, albeit with modified notation. Consider the special case where $\rho^g \supseteq \lambda^1$, $\sigma  = \rho^g / \lambda^1$ is connected, and $\sigma$ has exactly $g-1$ boxes. Consider only those pontableaux that have no left-removals (i.e. such that $\rho^n = \lambda^{n+1}$ for all $n$). Such pontableaux either have $g-1$ distinct augmentations (all in different boxes), or one right-removal and $g$ different augmentations. These correspond bijectively to the pontableaux considered in \cite{clpt}. The difference in notation is that we now indicated the removals by ``$n-$'' rather than ``$-n$,'' in order to distinguish left-removals and right-removals; this distinction is non-existent in \cite{clpt}, where only top-dimensional strata are labeled with pontableaux.
\end{remark}

\subsubsection{The poset structure}

Pontableaux (for a fixed choice of $\lambda^1$ and $\rho^g$, i.e. for fixed data $(g,r,d,\alpha,\beta)$) are arranged into a poset as follows.

\begin{defn} \label{def:ptposet}
Let $P,\overline{P}$ be two pontableaux in $\PT(\lambda^1,\rho^g)$. Say that $P$ \emph{generizes to} $\overline{P}$ (or that $\overline{P}$ \emph{specializes to} $P$), written $\overline{P} \supseteq P$, if for all $n \in \{1,2,\cdots,g-1\}$,
$$
\lambda^{n+1} \subseteq \overline{\lambda}^{n+1} \subseteq \overline{\rho}^n \subseteq \rho^n,
$$
where $\lambda^n$ and $\rho^n$ (respectively, $\overline{\lambda}^n$ and $\overline{\rho}^n$) are the tuples in the pontableau sequence of $P$ (respectively, $\overline{P}$). Regard  $\PT(\lambda^1,\rho^g)$ as a poset with this partial order.
\end{defn}

\begin{example} \label{ex:posets}
An example of the poset of pontableaux is shown in 
Figure \ref{fig:poset3}. If two or more strata share the same underlying set-valued tableau, they are enclosed by a dashed line. Note that the example in Figure \ref{fig:poset3} is the same example as \cite[Figure 7]{clpt}, except that now all strata are displayed, not just the top-dimensional ones.

\begin{figure}

\begin{tikzpicture}[xscale=0.6,yscale=0.75,transform shape]

\node (S1-1) at (18,-12){
\begin{tabular}{|c|c|}
\cline{1-2}
$\st 1$ & $\st 4$\\
\cline{1-2}
$\st 3$ & $\st 5$\\
\cline{1-2}
\end{tabular}
};
\node (S1-2) at (12,-12){
\begin{tabular}{|c|c|}
\cline{1-2}
$\st 1$ & $\st 4$\\
\cline{1-2}
$\st 2$ & $\st 5$\\
\cline{1-2}
\end{tabular}
};
\node (S1-3) at (6,-16){
\begin{tabular}{|c|c|c|}
\cline{1-3}
$\st 1$ & $\st 3$ & $\st 4,5-$\\
\cline{1-3}
$\st 2$ & $\st 5$ & \multicolumn{1}{|c}{}\\
\cline{1-2}
\end{tabular}
};
\node (S1-4) at (18,-16){
\begin{tabular}{|c|c|c|}
\cline{2-3}
\multicolumn{1}{c|}{} & $\st 1$ & $\st 4$\\
\cline{1-3}
$\st 1-,2$ & $\st 3$ & $\st 5$\\
\cline{1-3}
\end{tabular}
};
\node (S1-5) at (6,4){
\begin{tabular}{|c|c|c|}
\cline{2-3}
\multicolumn{1}{c|}{} & $\st 1$ & $\st 3$\\
\cline{1-3}
$\st 1-,2$ & $\st 4$ & $\st 5$\\
\cline{1-3}
\end{tabular}
};
\node (S1-6) at (12,4){
\begin{tabular}{|c|c|c|}
\cline{1-3}
$\st 1$ & $\st 2$ & $\st 3,4-$\\
\cline{1-3}
$\st 4$ & $\st 5$ & \multicolumn{1}{|c}{}\\
\cline{1-2}
\end{tabular}
};
\node (S1-7) at (24,-12){
\begin{tabular}{|c|c|}
\cline{1-2}
$\st 2$ & $\st 4$\\
\cline{1-2}
$\st 3$ & $\st 5$\\
\cline{1-2}
\end{tabular}
};
\node (S1-8) at (12,8){
\begin{tabular}{|c|c|c|}
\cline{1-3}
$\st 1$ & $\st 2$ & $\st 3,5-$\\
\cline{1-3}
$\st 4$ & $\st 5$ & \multicolumn{1}{|c}{}\\
\cline{1-2}
\end{tabular}
};
\node (S1-9) at (18,4){
\begin{tabular}{|c|c|c|}
\cline{1-3}
$\st 1$ & $\st 2$ & $\st 4,5-$\\
\cline{1-3}
$\st 3$ & $\st 5$ & \multicolumn{1}{|c}{}\\
\cline{1-2}
\end{tabular}
};
\node (S1-10) at (0,-12){
\begin{tabular}{|c|c|}
\cline{1-2}
$\st 1$ & $\st 3$\\
\cline{1-2}
$\st 2$ & $\st 4$\\
\cline{1-2}
\end{tabular}
};
\node (S1-11) at (6,-12){
\begin{tabular}{|c|c|}
\cline{1-2}
$\st 1$ & $\st 3$\\
\cline{1-2}
$\st 2$ & $\st 5$\\
\cline{1-2}
\end{tabular}
};
\node (S1-12) at (6,0){
\begin{tabular}{|c|c|}
\cline{1-2}
$\st 1$ & $\st 3$\\
\cline{1-2}
$\st 4$ & $\st 5$\\
\cline{1-2}
\end{tabular}
};
\node (S1-13) at (0,0){
\begin{tabular}{|c|c|}
\cline{1-2}
$\st 2$ & $\st 3$\\
\cline{1-2}
$\st 4$ & $\st 5$\\
\cline{1-2}
\end{tabular}
};
\node (S1-14) at (12,-4){
\begin{tabular}{|c|c|c|}
\cline{2-3}
\multicolumn{1}{c|}{} & $\st 1$ & $\st 2$\\
\cline{1-3}
$\st 2-,3$ & $\st 4$ & $\st 5$\\
\cline{1-3}
\end{tabular}
};
\node (S1-15) at (18,0){
\begin{tabular}{|c|c|}
\cline{1-2}
$\st 1$ & $\st 2$\\
\cline{1-2}
$\st 3$ & $\st 5$\\
\cline{1-2}
\end{tabular}
};
\node (S1-16) at (6,-6){
\begin{tabular}{|c|c|}
\cline{1-2}
$\st 1$ & $\st 3$\\
\cline{1-2}
$\st 2,3-,4$ & $\st 5$\\
\cline{1-2}
\end{tabular}
};
\node (S1-17) at (12,-8){
\begin{tabular}{|c|c|c|}
\cline{2-3}
\multicolumn{1}{c|}{} & $\st 1$ & $\st 2$\\
\cline{1-3}
$\st 1-,3$ & $\st 4$ & $\st 5$\\
\cline{1-3}
\end{tabular}
};
\node (S1-18) at (18,-6){
\begin{tabular}{|c|c|}
\cline{1-2}
$\st 1$ & $\st 2,3-,4$\\
\cline{1-2}
$\st 3$ & $\st 5$\\
\cline{1-2}
\end{tabular}
};
\node (S1-19) at (24,0){
\begin{tabular}{|c|c|}
\cline{1-2}
$\st 1$ & $\st 2$\\
\cline{1-2}
$\st 3$ & $\st 4$\\
\cline{1-2}
\end{tabular}
};
\node (S1-20) at (12,0){
\begin{tabular}{|c|c|}
\cline{1-2}
$\st 1$ & $\st 2$\\
\cline{1-2}
$\st 4$ & $\st 5$\\
\cline{1-2}
\end{tabular}
};
\node (S0-1) at (21,-12){
\begin{tabular}{|c|c|}
\cline{1-2}
$\st 1,-2,2$ & $\st 4$\\
\cline{1-2}
$\st 3$ & $\st 5$\\
\cline{1-2}
\end{tabular}
};
\node (S0-2) at (18,-3){
\begin{tabular}{|c|c|}
\cline{1-2}
$\st 1$ & $\st 2,-4,4$\\
\cline{1-2}
$\st 3$ & $\st 5$\\
\cline{1-2}
\end{tabular}
};
\node (S0-3) at (9,-12){
\begin{tabular}{|c|c|}
\cline{1-2}
$\st 1$ & $\st 3,-4,4$\\
\cline{1-2}
$\st 2$ & $\st 5$\\
\cline{1-2}
\end{tabular}
};
\node (S0-4) at (15,-12){
\begin{tabular}{|c|c|}
\cline{1-2}
$\st 1$ & $\st 4$\\
\cline{1-2}
$\st 2,-3,3$ & $\st 5$\\
\cline{1-2}
\end{tabular}
};
\node (S0-5) at (9,0){
\begin{tabular}{|c|c|}
\cline{1-2}
$\st 1$ & $\st 2,-3,3$\\
\cline{1-2}
$\st 4$ & $\st 5$\\
\cline{1-2}
\end{tabular}
};
\node (S0-6) at (3,-12){
\begin{tabular}{|c|c|}
\cline{1-2}
$\st 1$ & $\st 3$\\
\cline{1-2}
$\st 2$ & $\st 4,-5,5$\\
\cline{1-2}
\end{tabular}
};
\node (S0-7) at (6,-14){
\begin{tabular}{|c|c|c|}
\cline{1-3}
$\st 1$ & $\st 3$ & $\st 4,-5$\\
\cline{1-3}
$\st 2$ & $\st 5$ & \multicolumn{1}{|c}{}\\
\cline{1-2}
\end{tabular}
};
\node (S0-8) at (12,-2){
\begin{tabular}{|c|c|c|}
\cline{2-3}
\multicolumn{1}{c|}{} & $\st 1$ & $\st 2$\\
\cline{1-3}
$\st -3,3$ & $\st 4$ & $\st 5$\\
\cline{1-3}
\end{tabular}
};
\node (S0-9) at (3,0){
\begin{tabular}{|c|c|}
\cline{1-2}
$\st 1,-2,2$ & $\st 3$\\
\cline{1-2}
$\st 4$ & $\st 5$\\
\cline{1-2}
\end{tabular}
};
\node (S0-10) at (15,0){
\begin{tabular}{|c|c|}
\cline{1-2}
$\st 1$ & $\st 2$\\
\cline{1-2}
$\st 3,-4,4$ & $\st 5$\\
\cline{1-2}
\end{tabular}
};
\node (S0-11) at (18,-9){
\begin{tabular}{|c|c|}
\cline{1-2}
$\st 1$ & $\st 2,-3,4$\\
\cline{1-2}
$\st 3$ & $\st 5$\\
\cline{1-2}
\end{tabular}
};
\node (S0-12) at (6,2){
\begin{tabular}{|c|c|c|}
\cline{2-3}
\multicolumn{1}{c|}{} & $\st 1$ & $\st 3$\\
\cline{1-3}
$\st -2,2$ & $\st 4$ & $\st 5$\\
\cline{1-3}
\end{tabular}
};
\node (S0-13) at (6,-9){
\begin{tabular}{|c|c|}
\cline{1-2}
$\st 1$ & $\st 3$\\
\cline{1-2}
$\st 2,-4,4$ & $\st 5$\\
\cline{1-2}
\end{tabular}
};
\node (S0-14) at (6,-3){
\begin{tabular}{|c|c|}
\cline{1-2}
$\st 1$ & $\st 3$\\
\cline{1-2}
$\st 2,-3,4$ & $\st 5$\\
\cline{1-2}
\end{tabular}
};
\node (S0-15) at (12,6){
\begin{tabular}{|c|c|c|}
\cline{1-3}
$\st 1$ & $\st 2$ & $\st 3,-5$\\
\cline{1-3}
$\st 4$ & $\st 5$ & \multicolumn{1}{|c}{}\\
\cline{1-2}
\end{tabular}
};
\node (S0-16) at (12,2){
\begin{tabular}{|c|c|c|}
\cline{1-3}
$\st 1$ & $\st 2$ & $\st 3,-4$\\
\cline{1-3}
$\st 4$ & $\st 5$ & \multicolumn{1}{|c}{}\\
\cline{1-2}
\end{tabular}
};
\node (S0-17) at (12,-6){
\begin{tabular}{|c|c|c|}
\cline{2-3}
\multicolumn{1}{c|}{} & $\st 1$ & $\st 2$\\
\cline{1-3}
$\st -2,3$ & $\st 4$ & $\st 5$\\
\cline{1-3}
\end{tabular}
};
\node (S0-18) at (21,0){
\begin{tabular}{|c|c|}
\cline{1-2}
$\st 1$ & $\st 2$\\
\cline{1-2}
$\st 3$ & $\st 4,-5,5$\\
\cline{1-2}
\end{tabular}
};
\node (S0-19) at (18,2){
\begin{tabular}{|c|c|c|}
\cline{1-3}
$\st 1$ & $\st 2$ & $\st 4,-5$\\
\cline{1-3}
$\st 3$ & $\st 5$ & \multicolumn{1}{|c}{}\\
\cline{1-2}
\end{tabular}
};
\node (S0-20) at (18,-14){
\begin{tabular}{|c|c|c|}
\cline{2-3}
\multicolumn{1}{c|}{} & $\st 1$ & $\st 4$\\
\cline{1-3}
$\st -2,2$ & $\st 3$ & $\st 5$\\
\cline{1-3}
\end{tabular}
};
\draw[thick,black,->] (S1-1) -- (S0-1);
\draw[thick,black,->] (S1-1) -- (S0-4);
\draw[thick,black,->] (S1-1) -- (S0-11);
\draw[thick,black,->] (S1-1) -- (S0-20);
\draw[thick,black,->] (S1-2) -- (S0-3);
\draw[thick,black,->] (S1-2) -- (S0-4);
\draw[thick,black,->] (S1-3) -- (S0-7);
\draw[thick,black,->] (S1-4) -- (S0-20);
\draw[thick,black,->] (S1-5) -- (S0-12);
\draw[thick,black,->] (S1-6) -- (S0-15);
\draw[thick,black,->] (S1-6) -- (S0-16);
\draw[thick,black,->] (S1-7) -- (S0-1);
\draw[thick,black,->] (S1-8) -- (S0-15);
\draw[thick,black,->] (S1-9) -- (S0-19);
\draw[thick,black,->] (S1-10) -- (S0-6);
\draw[thick,black,->] (S1-11) -- (S0-3);
\draw[thick,black,->] (S1-11) -- (S0-6);
\draw[thick,black,->] (S1-11) -- (S0-7);
\draw[thick,black,->] (S1-11) -- (S0-13);
\draw[thick,black,->] (S1-12) -- (S0-5);
\draw[thick,black,->] (S1-12) -- (S0-9);
\draw[thick,black,->] (S1-12) -- (S0-12);
\draw[thick,black,->] (S1-12) -- (S0-14);
\draw[thick,black,->] (S1-13) -- (S0-9);
\draw[thick,black,->] (S1-14) -- (S0-8);
\draw[thick,black,->] (S1-14) -- (S0-17);
\draw[thick,black,->] (S1-15) -- (S0-2);
\draw[thick,black,->] (S1-15) -- (S0-10);
\draw[thick,black,->] (S1-15) -- (S0-18);
\draw[thick,black,->] (S1-15) -- (S0-19);
\draw[thick,black,->] (S1-16) -- (S0-13);
\draw[thick,black,->] (S1-16) -- (S0-14);
\draw[thick,black,->] (S1-17) -- (S0-17);
\draw[thick,black,->] (S1-18) -- (S0-2);
\draw[thick,black,->] (S1-18) -- (S0-11);
\draw[thick,black,->] (S1-19) -- (S0-18);
\draw[thick,black,->] (S1-20) -- (S0-5);
\draw[thick,black,->] (S1-20) -- (S0-8);
\draw[thick,black,->] (S1-20) -- (S0-10);
\draw[thick,black,->] (S1-20) -- (S0-16);

\draw[dashed,blue] (4.5,0.8) -- (7.5,0.8) -- (7.5,5) -- (4.5,5) -- cycle;
\draw[dashed,blue] (4.5,-1.8) -- (7.5,-1.8) -- (7.5,-10.2) -- (4.5,-10.2) -- cycle;
\draw[dashed,blue] (4.5,-12.8) -- (7.5,-12.8) -- (7.5,-17) -- (4.5,-17) -- cycle;

\draw[dashed,blue] (10.5,0.8) -- (13.5,0.8) -- (13.5,9) -- (10.5,9) -- cycle;
\draw[dashed,blue] (10.5,-0.8) -- (13.5,-0.8) -- (13.5,-9) -- (10.5,-9) -- cycle;

\draw[dashed,blue] (16.5,0.8) -- (19.5,0.8) -- (19.5,5) -- (16.5,5) -- cycle;
\draw[dashed,blue] (16.5,-1.8) -- (19.5,-1.8) -- (19.5,-10.2) -- (16.5,-10.2) -- cycle;
\draw[dashed,blue] (16.5,-12.8) -- (19.5,-12.8) -- (19.5,-17) -- (16.5,-17) -- cycle;

\end{tikzpicture}
\caption{The pontableau poset for $(g,r,d,\alpha,\beta) = (5,1,4,(0,0),(0,0))$. Compare to \cite[Figure 7]{clpt}.} \label{fig:poset3}
\end{figure}
\end{example}

Define the following three combinatorial attributes of pontableaux.

\begin{defn} \label{def:pontFuncs}
Let $P$ be a pontableau. Define $\mu(P), \dim P$, and $\chi(P)$ as follows.
\begin{enumerate}
\item Denote by $L$ the number of left removals in $P$. Let $\mu(P)$ be $0$ if $P$ has two identical left removals (i.e. ``$-n$'' for the same value of $n$) in horizontally or vertically adjacent boxes. Otherwise, let $\mu(P) = (-1)^L$.
\item Let $\dim P =  g - \# \left( \mbox{augmentations in $P$}\right) + \# \left( \mbox{right-removals in $P$} \right)$.
\item Let $\chi(P) = 1$ if every possible augmentation $1,2,\cdots,g$ occurs somewhere in $P$, and let $\chi(P) = 0$ otherwise.
\end{enumerate}
\end{defn}

\subsubsection{Properties of the M\"obius function of the pontableau poset}

We first verify that $\mu(P)$ indeed gives the M\"obius function for the poset $\PT(\lambda^1,\rho^g)$.

\begin{lemma} \label{lemma:ptmu}
Fix data $(\lambda^1,\rho^g)$. The function $\mu$ given in Definition \ref{def:pontFuncs}, restricted to the set $\PT(\lambda^1,\rho^g)$, is equal to the M\"obius function $\mu_{\PT(\lambda^1\rho^g)}$ (see Definition \ref{def:our-mobius}).
\end{lemma} 

\begin{proof}
Fix a pontableau $P \in \PT(\lambda^1,\rho^g)$. It suffices to verify the equation
$$
\sum_{\overline{P} \supseteq P} \mu(\overline{P}) = 1,
$$
where the sum is taken over all $\overline{P} \in \PT(\lambda^1,\rho^g)$ generizing $P$ (including $P$ itself).

The $g-1$ chains of inclusions $\lambda^{n+1} \subseteq \overline{\lambda}^{n+1} \subseteq \overline{\rho}^n \subseteq \rho^n$ in Definition \ref{def:ptposet} are independent of each other, and if $(\overline{\lambda}^1,\cdots,\overline{\rho}^g)$ satisfy these $g-1$ chains of inclusions, then they are a pontableau sequence. So the choice of $\overline{P}$ amounts to $g-1$ independent choices of a pair $(\overline{\rho}^n,\overline{\lambda}^{n+1})$. 

Let $\rho \supseteq \lambda$ be two nonincreasing tuples. We use the following notation:

\begin{eqnarray*}
I(\rho,\lambda) &=& \{\rho':\ \rho \supseteq \rho' \supseteq \lambda\}\\
M(\rho,\lambda) &=& \{(\rho',\lambda'):\ \rho \supseteq \rho' \supseteq \lambda' \supseteq \lambda \}\\
f(\rho,\lambda) &=& \begin{cases} (-1)^{|\rho / \lambda|} & \mbox{ if no two boxes of $\rho / \lambda$ are adjacent}\\ 0 & \mbox{ otherwise.}\end{cases}\\
s(\rho,\lambda) &=& \sum_{(\rho',\lambda') \in M(\rho,\lambda)} f(\rho',\lambda')
\end{eqnarray*}

The set of all $\overline{P} \supseteq P$ is in bijection with $\prod_{n=1}^{g-1} M(\rho^n,\lambda^{n+1})$, and the set of left removals ``$-n$'' in $\overline{P}$ is in bijection with the boxes of $\rho^{n-1} / \lambda^n$. It follows that $\mu(\overline{P}) = \prod_{n=1}^{g-1} f(\overline{\rho}^n, \overline{\lambda}^{n+1})$. Therefore
 $$
 \sum_{\overline{P} \supseteq P} \mu(\overline{P}) = \prod_{g=1}^{n-1} s(\rho^n,\lambda^{n+1}).
 $$

The lemma will immediately follow from this equation and the following claim.

\textit{Claim.} For any two tuples $\rho \supseteq \lambda$, $s(\rho,\lambda) = 1$.

\textit{Proof of claim.} Rearranging the summation,
$$
s(\rho,\lambda) = \sum_{\rho' \in I(\rho,\lambda)} \sum_{\lambda' \in I(\rho',\lambda)} f(\rho',\lambda').
$$
In the inner sum, the only $\lambda'$ that give nonzero values of $f(\rho',\lambda')$ are obtained by adding to the boxes of $\rho'$ some subset of the set $C = \{ \mbox{inward corners of } \rho'\} \cap \lambda$. Conversely, any subset $S \subseteq C$ gives a choice $\lambda'$ that contributes $(-1)^S$ to the inner sum. Therefore, for fixed $\rho',\lambda$,
$$
\sum_{\lambda' \in I(\rho',\lambda)} f(\rho',\lambda') = \sum_{S \subseteq C} (-1)^{|S|} = 
\begin{cases}
1 & \mbox{ if } C = \emptyset\\
0 & \mbox{ otherwise.}
\end{cases}
$$
Therefore the inner sum is equal to $1$ if and only if $\rho' = \lambda$, and $0$ otherwise. The claim follows, and also the lemma.
\end{proof}

Recall the definition of \emph{almost standard} from Section \ref{ss:svt}. In what follows, we will always be refering to almost-standard tableaux with content a subset of $\{1,2,\cdots,g\}$.

\begin{lemma} \label{lemma:sumGivenTab}
Let $t$ be any set-valued tableau (not necessarily almost-standard) whose content is an $N$-element subset of $\{1,2,\cdots,g\}$, and denote by $\tab^{-1}(t,g,\lambda^1,\rho^g)$ the set of all pontableaux $P \in \PT(\lambda^1,\rho^g)$ such that $\tab(P) = t$. Let $\sigma$ be the skew shape $\rho^g / \lambda^1$. Then
$$
\sum_{P \in \tab^{-1}(t,g,\lambda^1,\rho^g)} \mu(P) =
\begin{cases}
(-1)^{N-|\sigma|} & \mbox{ if $t$ is an almost-standard set-valued tableau on $\sigma$}\\
0 & \mbox{ otherwise.}
\end{cases}
$$
\end{lemma}

\begin{proof}
We may assume throughout that \emph{any given symbol $n$ occurs at most once in $t$,} since otherwise $t$ is not the underlying set-valued tableau of any pontableau, and both sides of the claimed equation are equal to $0$. We will denote by $b_n$ the box in which the label $n$ occurs in $t$, if it does occur; if $n$ does not occur in $t$ we will say that $b_n$ does not exist. We may also assume that \emph{either $b_g$ does not exist, or it is an outward corner of $\rho^g$}. This is because if $b_g$ exists but is not an outward corner of $\rho^g$, then there are no almost-standard set-valued tableaux on $\sigma$ with $g$ in box $b_g$, and also there are no pontableaux with an augmentation ``$g$'' in box $b_g$, and again both sides of the claimed equation are $0$.

We proceed by induction on $g$. Consider first the case $g=1$. In this case, $\PT(\lambda^1,\rho^1)$ is either empty or contains a single element, depending on whether $\sigma = \rho^1 / \lambda^1$ has more than one box. If it is nonempty, then its single element $P$ has $\mu(P) = 1$, and there is a single almost-standard set-valued tableau on $\sigma$: either the empty tableau or the tableau obtained by placing ``$1$'' in the single box of $\sigma$. If the poset is empty, then $\sigma$ has no almost-standard set-valued tableau, since there is only one label available for more than one box. So the lemma holds for $g=1$.

Now suppose that $g \geq 2$, and that the lemma holds for smaller values of $g$. Fix a set-valued tableau $t$ with symbols chosen from $\{1,2,\cdots,g\}$, and let $t'$ be the set-valued tableau obtained by removing $g$ from $t$ if it appears. As observed in the first paragraph, we may assume that $t$ has no repeated symbols, and the last symbol $g$, if it appears, appears in an outward corner of $\rho^g$.

Reorder the sum in question according to the choice of $\rho^{g-1}$ and $\lambda^{g}$. Here, each sum can be taken over the set of all possible $(r+1)$-tuples (only finitely many terms will be nonzero).

\begin{eqnarray*}
\sum_{P \in \tab^{-1}(t,g,\lambda^1,\rho^g)} \mu(P) = \sum_{\rho^{g-1}} \ \sum_{\lambda^g} \{ \mu(P):\ P \mbox{ has specified choice of $\rho^{g-1},\lambda^g$}\}
\end{eqnarray*}

Define the function $f$ as in the proof of Lemma \ref{lemma:ptmu}, and observe that if $P = (\lambda^1,\cdots,\rho^g) \in \tab^{-1}(t,g,\lambda^1,\rho^g)$, then defining $P' = (\lambda^1,\cdots,\rho^{g-1})$, we have $P' \in \PT(\lambda^1,\rho^{g-1})$, $\mu(P) = \mu(P') f(\rho^{g-1}, \lambda^g)$, and the underlying set-valued tableau of $P'$ is $t'$.

Therefore the sum may be rewritten as follows. Here, the sums are taken over the following sets: $\rho^{g-1}$ is chosen from the set of all nonincreasing $(r+1)$-tuples; $P'$ is chosen from $\tab^{-1}(t',g-1,\lambda^1,\rho^{g-1})$; $\lambda^{g}$ is chosen from either $I(\rho^{g-1} \backslash b_g, \rho^g \backslash b_g)$ (if $b_g$ exists, i.e. $g$ occurs in $t$), or $I(\rho^{g-1}, \rho^g)$ (if $b_g$ does not exist). This is because $\rho^g \subseteq \lambda^g \subseteq \rho^{g-1}$ and $b_g\not\in\lambda^g$, if $b_g$ exists.

$$
\sum_P \mu(P) = \sum_{\rho^{g-1}} \left( \sum_{P'} \mu(P') \right) \left( \sum_{\lambda^g} f(\rho^{g-1},\lambda^g) \right)
$$

\textit{Case 1:} $b_g$ does not exist. In this case, we can conclude, as in the proof of Lemma \ref{lemma:ptmu}, that the second inner sum is
$$
\sum_{\lambda^g} f(\rho^{g-1},\lambda^g) =
\begin{cases}
1 & \mbox{ if } \rho^{g-1} = \rho^g\\
0 & \mbox{ otherwise.}
\end{cases}
$$

Therefore the overall sum reduces to only the term where $\rho^{g-1} = \rho^g$, hence it is $\sum_{P'} \mu(P')$, where the sum is taken over $P' \in \tab^{-1}(t',g-1,\lambda^1,\rho^{g})$. The statement of the lemma now follows from the inductive hypothesis.

\textit{Case 2:} $b_g$ is an outward corner of $\rho^g$.

The second inner sum $\sum_{\lambda^g} f(\rho^{g-1},\lambda^g)$ has one nonzero term for each choice of a set $C$ of outward corners of $\rho^{g-1}$ such that 
\begin{itemize}
\item If $b_g\in \rho^{g-1}$ then $C$ is contained in the set of outward corners between $\rho^{g-1}\setminus b_g$ and $\rho^g\setminus b_g$,
\item if $b_g\not\in\rho^{g-1}$ then $C$ is contained in the set of outward corners between $\rho^{g-1}$ and $\rho^g\setminus b_g$.
\end{itemize}

The contribution of this term is $(-1)^{|C|}$. These terms will cancel each other unless there is only one of them. 
Hence, if $b_g \in\rho^{g-1}$, then $\rho^{g-1}\setminus b_g = \rho^g\setminus b_g$ and the inner sum is $f(\rho^{g-1},\rho^{g-1}\setminus b_g) = -1$.  If $b_g\not\in \rho^{g-1}$, then $\rho^{g-1} = \rho^g \setminus b_g$ and the inner sum is $f(\rho^{g-1},\rho^{g-1}) = 1$.  Therefore,
it follows that the inner sum is

$$
\sum_{\lambda^g} f(\rho^{g-1},\lambda^g) = 
\begin{cases}
1 & \mbox{ if } \rho^{g-1} = \rho^g \backslash b_g\\
-1 & \mbox{ if } \rho^{g-1} = \rho^g\\
0 & \mbox{ otherwise}.
\end{cases}
$$

Therefore the overall sum is the following difference of two terms involving posets of pontableaux for $g-1$.

$$
\sum_P \mu(P) = \sum_{P' \in \tab^{-1}(t,g-1,\lambda^1,\rho^g \backslash b_g)} \mu(P') - \sum_{P' \in \tab^{-1}(t,g-1,\lambda^1,\rho^g)} \mu(P')
$$

By the inductive hypothesis, the two sums on the right side are as follows.

$$
\sum_{P' \in \tab^{-1}(t,g-1,\lambda^1,\rho^g \backslash b_g)} \mu(P') =
\begin{cases}
(-1)^{(N-1) - (|\sigma|-1)} & \mbox{ if $t'$ is almost-standard on $\sigma \backslash b_g$}\\
0 & \mbox{ otherwise.}
\end{cases}
$$

$$
\sum_{P' \in \tab^{-1}(t,g-1,\lambda^1,\rho^g)} \mu(P') =
\begin{cases}
(-1)^{(N-1) - |\sigma|} & \mbox{ if $t'$ is almost-standard on $\sigma$}\\
0 & \mbox{ otherwise.}
\end{cases}
$$

Now, note that $t$ is almost-standard on $\sigma$ if and only if either $g$ is the only symbol in its box and $t'$ is almost-standard on $\sigma \backslash b_g$ or $g$ is not the only symbol in its box and $t'$ is almost-standard on $\sigma$. In either case, one of the sums above is zero and the other is $\pm 1$, and their difference is the desired quantity from the lemma statement. On the other hand, if $t$ is not almost-standard on $\sigma$, then both of the sums above are zero and again the lemma statement follows.
This completes the induction, and establishes the lemma.
\end{proof}

From these two lemmas, Theorems \ref{prop:ptmuchi} and \ref{thm:grdab-chain-euler} follow.

\begin{proof}[Proof of Theorem \ref{prop:ptmuchi}]
Rearrange the summation to group together pontableaux with the same underlying tableau $t$. Note that for a pontableau $P$, the value $\chi(P)$ depends only on the underlying set-valued tableau $t$; hence we can denote this value by $\chi(t)$. Therefore we may write
$$
\sum_{P \in \PT(g,r,d,\alpha,\beta)} \mu(P) \chi(P) = \sum_{t} \chi(t) \sum_{P \in \tab^{-1}(t)} \mu(P).
$$
Since $\chi(t) = 1$ if $t$ has content $\{1,2,\cdots,g\}$, and $\chi(t) = 0$ otherwise, we need only sum over the set-valued tableaux on $\sigma$ with content $\{1,2,\cdots,g\}$. It follows from Lemma \ref{lemma:sumGivenTab} that this sum is equal to $(-1)^{g-|\sigma|}$ times the number of almost-standard set-valued tableaux on $\sigma$ with content $\{1,\ldots,g\}$.
\end{proof}

\subsection{Proof of Theorems \ref{thm:grdab-chain-structure} and \ref{thm:grdab-chain-euler}} \label{ss:ehscheme-proofs}

We now assemble the results above to describe in detail the geometry of the Eisenbud-Harris scheme $\Grdab(X,p,q)$, where $X$ is a generic twice-marked chain of elliptic curves. Throughout this subsection, fix data $(g,r,d,\alpha,\beta)$ and the chain $(X,p,q)$.

First note that, in light of Observation \ref{obs:pontcompat}, the nonempty loci $C(\bs \alpha, \bs \beta)$ are in bijection with pontableaux $P \in \PT(g,r,d,\alpha,\beta)$. Hence we will denote by $$C(T) \subseteq \Grdab(X,p,q)$$ the locus corresponding to a pontableau $T$. First we point out that the geometric facts about $C(\bs \alpha, \bs \beta)$ translate to combinatorial attributes of $T$.

\begin{lemma} \label{lemma:ct-attributes}
For any pontableau $T \in \PT(g,r,d,\alpha,\beta)$,
\begin{enumerate}
\item $C(T)$ is nonempty and equidimensional of dimension $\dim(T)$.
\item $\chi(C(T)) = \chi(T)$.
\item If $T$ has no left-removals, then a dense open subset of $C(T)$ consists of refined series. Otherwise, all points of $C(T)$ correspond to coarse series.
\item The M\"obius function of the poset of loci $C(T)$ is equal to $\mu(P)$.
\end{enumerate}
\end{lemma}
\begin{proof}
The nonemptiness claim in Part (1) follows from Observation \ref{obs:pontcompat} and part (1) of Lemma \ref{lemma:cab-attributes}. For the dimension claim, note that part (2) of Lemma \ref{lemma:cab-attributes} and the definition of $\alpha^{n}_i, \beta^n_i$ in terms of $T$ shows that
$$
\dim C(T) = \rho - \sum_{n=1}^{g-1} \sum_{i=0}^r (\rho^n_{r-i} - \lambda^{n+1}_i),
$$
where $\rho$ (with no subscripts or superscripts) here denotes the Brill-Noether number, rather than the numbers $\rho^n_i$ encoded in $T$. In other words, $\dim C(T)$ is $\rho$ minus the number of left-removals in $T$. Expressing $\rho$ in terms of $\lambda^1_{i}$ and $\rho^g_i$ shows that it is equal to $g$ minus the number of augmentations in $T$ plus the number of removals (left or right). Hence $\dim C(T)$ is equal to $g$ plus the number of \textit{right}-removals in $T$ minus the number of augmentations in $T$, which is $\dim T$. This proves part (1).

Part (2) follows from part (5) of Lemma \ref{lemma:cab-attributes} and the observation that for any value of $n$, the equality $\alpha^n_i + \beta^n_{r-i} = d-r$ holds for some $i$ if and only if the augmentation ``$n$'' appears in $T$.

Part (3) follows from part (4) of Lemma \ref{lemma:cab-attributes} and the observation that $\bs \beta$ is complementary to $\bs \alpha$ if and only if there are no left-removals in $P$.

Part (4) follows from part (3) of Lemma \ref{lemma:cab-attributes}, the observation that the condition stated there matches the definition of the poset structure on $\PT(g,r,d,\alpha,\beta)$, and Lemma \ref{lemma:ptmu}.
\end{proof}

\begin{proof}[Proof of Theorem \ref{thm:grdab-chain-structure}]
Note that $\hat{\rho} \geq 0$ if and only if $|\sigma| \leq g$. If $|\sigma| \leq g$, then it is possible to construct a pontableau for the data $(g,r,d,\alpha,\beta)$: the boxes of $\sigma$ can be filled in with any almost-standard tableaux, and then a right-removal ``$g-$'' can be placed in all boxes of $\lambda^1 / \rho^g$. The result will be a pontableau. Hence by Lemma \ref{lemma:ct-attributes} part (1), $\Grdab(X,p,q)$ is nonempty. Conversely, if $\Grdab(X,p,q)$ is nonempty, then by Lemma \ref{lemma:vsunion} and the fact that nonempty loci $C(\bs \alpha, \bs \beta)$ correspond to pontableaux, there exists a pontableau for the data $(g,r,d,\alpha,\beta)$, hence an almost-standard set-valued tableau on $\sigma$. Hence $\sigma$ has at most $g$ boxes, and $\hat{\rho} \geq 0$.

By Lemma \ref{lemma:vsunion}, $\Grdab(X,p,q)$ is a union of reduced schemes of pure dimension $\rho$, namely $C(T)$ for pontableaux $T$ with no left-removals. So it too has pure dimension $\rho$. Each $C(T)$ has a dense open subset of refined series, hence so does $\Grdab(X,p,q)$ as a whole.
\end{proof}

\begin{proof}[Proof of Theorem \ref{thm:grdab-chain-euler}]
Lemma \ref{lemma:ct-attributes}, combined with Proposition \ref{prop:cut-n-paste} and Theorem \ref{prop:ptmuchi}, show that
\begin{eqnarray*}
\chi(\Grdab(X,p,q)) &=& \sum_{P \in \PT(g,r,d,\alpha,\beta)} \mu(P) \chi(P)\\
&=& (-1)^{g-|\sigma|} \# (\mbox{standard set-valued tableaux on $\sigma$ of content $\{1,\ldots,g\}$}).
\end{eqnarray*}
\end{proof}

\subsection{Proof of the main theorem} \label{ss:mainthm-proof}

We now deduce our main result, Theorem \ref{thm:main}, from properties about smoothing of limit linear series, together with the analogous statement (Theorem \ref{thm:grdab-chain-euler}) for chains of elliptic curves.

\begin{proof}[Proof of Theorem \ref{thm:main}]
Let $(X,p,q)$ be a general twice-marked curve of genus $g$. By Theorem \ref{thm:osserman-rhohat}, $\Grdab(X,p,q)$ is nonempty if and only if $\hat{\rho} \geq 0$, or equivalently $|\sigma| \leq g$. This is equivalent to the existence of a set-valued tableau on $\sigma$ with content $\{1,\ldots,g\}$. 
The theorem holds vacuously in case $\hat{\rho} < 0$, so we assume that $\hat{\rho} \geq 0$.

By semicontinuity, there is a dense open subset $\mathcal{U}$ of $\mathcal{M}_{g,2}$ on which $\chi(\Grdab(X,p,q))$ is constant. Let $(X_0,p_0,q_0)$ be a generic twice-marked elliptic chain, as defined in \ref{sec:prelim-lls}, and let $B$ be the spectrum of a discrete valuation ring. Then there exists a flat deformation of $(X_0,p_0,q_0)$ with base $B$ such that the induced morphism $B \rightarrow \overline{\mathcal{M}}_{g,2}$ sends the generic point to a point in $\mathcal{U}$. Replacing $B$ with a finite base extension if necessary, we may assume that the family of curves over $B$ is a smoothing family in the sense of \cite[Definition 3.9]{osserman-limit-not-compact-type}. Denote the general member of this family by $(X_\eta,p_\eta,q_\eta)$.

Theorem \ref{thm:grdab-chain-structure} shows that the hypotheses of \cite[Corollary 3.3]{murray-osserman} are satisfied, hence there exists a flat proper scheme over $B$ whose special fiber is the Eisenbud-Harris space on $(X_0,p_0,q_0)$, and whose general fiber is $G^{r,\alpha,\beta}_d(X_\eta,p_\eta,q_\eta)$. By flatness, the Euler characteristic of the structure sheaf of the general fiber is equal to that of the special fiber. Since the generic point of $B$ is sent to $\mathcal{U}$, it follows that this is also the Euler characteristic of the structure sheaf of $G^{r,\alpha,\beta}_d(X,p,q)$ for a general twice-marked curve $(X,p,q)$.
\end{proof}

\begin{remark}
The results of \cite{murray-osserman} that we use above are stated in terms of varieties $G^r_d(X)$ without marked points. However, similar arguments apply to the situation where marked points are present, as has already been noted by Osserman. 
\end{remark}

\noindent {\bf Acknowledgments.}  This work was influenced by conversations with a number of people, to whom we are grateful. We thank Greta Panova for suggesting to us that we look into set-valued tableaux.  Thanks to the American Institute of Mathematics and to Brian Osserman and Ravi Vakil, organizers of the AIM workshop on degenerations in algebraic geometry.  Ravi Vakil helped us crystallize the idea of the main theorem at that workshop. Thanks are also due to Brian Osserman for providing numerous useful references.  We thank Montserrat Teixidor i Bigas for introducing us to some aspects of this topic and for conversations in the early stages of our work. MC was supported by an NSA Young Investigators Grant, NSF DMS-1701924, and a Sloan Research Fellowship.

\bibliographystyle{amsalpha}
\bibliography{./my}
\end{document}